\documentclass[11pt,a4paper]{amsart}
\usepackage{amsmath}
\usepackage{amssymb,amsthm,amscd}
\usepackage{fullpage}
\usepackage{hyperref}
\usepackage{graphicx}

\newtheorem{theorem}{Theorem}

\theoremstyle{definition}
\newtheorem{definition}[theorem]{Definition}

\theoremstyle{remark}

\def\Aut{\operatorname{Aut}}
\def\SAut{\operatorname{SAut}}
\def\LND{\operatorname{LND}}
\def\HLND{\operatorname{HLND}}
\def\Pic{\operatorname{Pic}}

\def\AffCone{\operatorname{AffCone}}
\def\Ample{\operatorname{Ample}}
\def\Eff{\operatorname{\overline{NE}}}

\def\supp{\operatorname{supp}}
\def\Bs{{\mathrm{Bs}}}
\def\PP{{\mathbb P}}
\def\ZZ{{\mathbb Z}}
\def\NN{{\mathbb N}}
\def\G{{\mathbb G}}
\def\A{{\mathbb A}}

\def\K{{\mathbb K}}
\def\U{{\mathcal U}}

\def\bangle#1{\langle #1 \rangle}

\begin{document}

\tolerance=500
\author{A.Yu. Perepechko}
\title{Flexibility of affine cones over del Pezzo surfaces of degree 4 and 5}
\address{Department of Higher Algebra, Faculty of Mechanics and Mathematics, Moscow State University, Leninskie gory GSP-1, Moscow 119991, Russia}
\address{Universit\'e Grenoble I, Institut Fourier, UMR 5582 CNRS-UJF, BP 74, 34802 St Martin d'H\`eres c\'edex, France}
\email{perepechko@mccme.ru}
\thanks{This work was supported by The Ministry of education and science of Russian Federation, project 8214, in part by the Simons Foundation, and by RFBR grant nos. 12-01-00704 and 12-01-31342mol a.}

\begin{abstract}
We prove that the action of the special automorphism group on affine cones over del Pezzo surfaces of degree 4 and 5 is infinitely transitive.
\end{abstract}

\maketitle
\section{Introduction}
An affine algebraic variety $X$ defined over an algebraically closed field $\K$ of characteristic zero is called \emph{flexible} if the tangent space of $X$  at any smooth point is spanned by the tangent vectors to the orbits of one-parameter unipotent group actions \cite{flex}. In this paper we establish flexibility of affine cones over del Pezzo surfaces of degree 4 and 5.

It is well known that every effective action of one-dimensional unipotent group $\G_a=\G_a(\K)$ on~$X$ defines a locally nilpotent derivation $\delta\in\LND(\K[X])$ of the algebra of regular functions on $X$. All such actions generate a subgroup of \emph{special automorphisms} $\SAut X
\subset\Aut X$.

 A group $G$ is said to act on a set $S$ \emph{infinitely transitively} if it acts transitively on the set of $m$-tuples of pairwise distinct points in $S$ for any $m\in\NN$.

 The following theorem explains the significance of the flexibility concept.
\begin{theorem}[{\cite[Theorem 0.1]{flex}}]
  Let $X$ be an affine algebraic variety of dimension $\ge2$. Then the following conditions are equivalent:
  \vspace{-2pt}
  \begin{enumerate}
    \item The variety $X$ is flexible;
    \item the group $\SAut X$ acts transitively on the smooth locus $X_{\mathrm{reg}}$ of $X$;
    \item the group $\SAut X$ acts infinitely transitively on $X_{\mathrm{reg}}$.
  \end{enumerate}
\end{theorem}

  Three classes of flexible affine varieties are described in \cite{ArK}, namely affine cones over flag varieties, non-degenerate toric varieties of dimension $\ge2$, and suspensions over flexible varieties.
  Note that affine cones over del Pezzo varieties of degree $\ge6$ are toric, thereby they are flexible.

In this paper we consider cases of degree 4 and 5. In case of degree 5 we prove flexibility of affine cones corresponding to polarizations defined by arbitrary very ample divisors, whereas for degree 4 we prove flexibility only for certain very ample divisors, the anticanonical one included.   As for del Pezzo surfaces of degree $\le3$, it is proven the non-existence of any $\G_a$-actions on the affine cones over plurianticanonical embeddings, see \cite[Theorem~1.1]{CPW} for the case of degree 3 and \cite[Corollary~1.8]{KPZ4} for the case of degree $\le2$. 

In the proof we use the construction from \cite{Z}, which allows to associate a regular $\G_a$-action on an affine cone over a projective variety $Y$ to every open cylindrical subset of $Y$ of a special form. In Theorem \ref{flex} we provide a criterion of flexibility of an affine cone over a projective variety in terms of a transversal cover by such cylinders. We apply it to del Pezzo surfaces.

The author is grateful to M.G.~Zaidenberg for posing the problem and numerous discussions, to I.V.~Arzhantsev for useful remarks, and to the reviewer for corrections.

\section{Flexibility of affine cones}
 Let $Y$ be a projective variety and $H$ be a very ample divisor on $Y$. A polarization of $Y$ by $H$ provides an embedding $Y\hookrightarrow\PP^n$. Consider an affine cone $X=\AffCone_H Y\subset\A^{n+1}$ with vertex at the origin $0\in\AA^{n+1}$ corresponding to this embedding. There is a natural homothety action of the multiplicative group $\G_m=\G_m(\K)$ of the field $\K$ on $X$.
It defines a grading on the algebra $\K[X]$. A derivation on $\K[X]$ is called \emph{homogeneous} if it sends homogeneous elements into homogeneous ones. A subset of all homogeneous locally nilpotent derivations is denoted by $\HLND(\K[X])$. 
\begin{definition}[{\cite[Definitions 3.1.5, 3.1.7]{Z}}]
 We say that an open subset $U$ of a variety $Y$ is a \emph{cylinder} if $U\cong Z\times \A^1$, where $Z$ is a smooth variety with $\Pic Z=0$.   Given a divisor $H\subset Y$, we say that a cylinder $U$ is \emph{$H$-polar} if $U=Y\setminus \supp D$ for some effective divisor $D\in|d H|$, where $d>0$.
\end{definition}
\begin{definition}
  We call a subset $W\subset Y$ \emph{invariant} with respect to
   a cylinder $U=Z\times \A^1$ if $W\cap U=\pi_1^{-1}(\pi_1(W))$, where $\pi_1\colon U\to Z$ is the first projection of the direct product. In other words, every $\A^1$-fiber of the cylinder is either contained in $W$ or does not meet $W$.
\end{definition}
\begin{definition}
  We say that a variety $Y$ is \emph{transversally covered} by 
  cylinders $U_i$, $i=1,\ldots,s$, if $Y=\bigcup U_i$ and there is no proper subset $W\subset Y$ invariant with respect to all $U_i$.
\end{definition}
Clearly, every cylinder $U_i$ is smooth. Thus, a singular variety $Y$ does not admit a transversal covering by cylinders. It is also clear that $\dim Y\ge1$. 
The following theorem gives a criterion of flexibility for the affine cone over a projective embedding $Y\hookrightarrow\PP^n$ corresponding to the polarization by $H$.
\begin{theorem}\label{flex}
  If for some very ample divisor $H$ on a smooth projective variety $Y$ there exists a transversal covering by $H$-polar cylinders,
   then the affine cone $X=\AffCone_H Y$ is flexible.
\end{theorem}
\begin{proof}
The statement is obvious for $X=\A^{n+1}$. Thus, we may suppose that the vertex of the cone is the only singular point.

  By \cite[Theorem 3.1.9]{Z} for every cylinder on $Y$ there corresponds a homogeneous $\G_a$-action on $X$. Note from the explicit construction in \cite[Proposition 3.1.5]{Z} that the projection  $\pi\colon X^\times=X\setminus\{0\}\to Y$ sends the orbits of this action to fibers of the cylinder on $Y$, and the subset of fixed points on $X$ is the preimage of the cylinder complement.

Let $G\subset \SAut X$ be a subgroup generated by corresponding $\G_a$-actions. Consider an orbit $Gx$ of some point $x\in X^\times$. The image $\pi(Gx)\subset Y$ is invariant w.r.t. all covering cylinders. The transversality condition implies $\pi(Gx)=Y$. Since the group $G$ is generated by homogeneous actions, the natural $\G_m$-action on $X$ by homotheties normalizes the $G$-action on $X$ and sends $G$-orbits to $G$-orbits. Thus, $X^\times$ is a union of $G$-orbits, which projections coincide with $Y$.  Hence $X^\times=\bigcup_{\lambda\in\G_m}\lambda Gx$, where all $G$-orbits are closed in $X^\times$.

 Let us show that there exists the only open $G$-orbit $Gx=X^\times$. Assume the contrary. Then $\dim Gx = \dim Y$ and the stabilizer $S\subset\G_m$ of the orbit $Gx$ is finite. 
 Moreover, since the action of $S$ on $Gx$ is free, for any point $x^\prime\in Gx$ the intersection $Gx\cap \G_m x^\prime$ is an $S$-orbit consisting of $|S|$ distinct points.
  The blow up of $X$ in the origin is the total space of the linear bundle $[-H]$ on $Y$. It has a natural completion --- a  $\PP^1$-bundle $\hat X\to Y$.
  For a general point $x^\prime\in Gx$ the intersection  $\overline{Gx}\cap \overline{\G_m x^\prime}$, where $\overline{Z}$ denotes the closure of $Z\subset X^\times$ in $\hat X$, coincides with the orbit $Sx$. So, the intersection index $\overline{Gx}\cdot \overline{\G_m x^\prime}$ equals $|S|$. Since the intersection index is constant, for any point $x^\prime\in Gx$ there holds $\overline{Gx}\cap \overline{\G_m x^\prime}=Sx^\prime\subset X^\times$. Therefore, a quasi-affine variety $X^\times$ contains a projective one  $\overline{Gx}$, which is a contradiction.
    So, the group $G$ acts on $X^\times$ transitively.
\end{proof}

\section{Del Pezzo surface of degree 5}
Let $Y$ be a del Pezzo surface of degree 5. It is obtained by blowing up the projective plane $\PP^2$ in four points $P_1,\ldots,P_4$, no three of which are collinear \cite[Theorem IV.2.5]{Manin}. Since the automorphism group of the projective plane acts transitively on such 4-tuples of points, such a surface is unique up to isomorphism.
\begin{theorem}\label{theo5}
    Let $H$ be an arbitrary very ample divisor on the del Pezzo surface $Y$ of degree~5. Then the corresponding affine cone $\AffCone_H Y$ is flexible.
\end{theorem}
The proof proceeds in several steps, see Sections \ref{3.1} and \ref{3.2}.
We let $E_i$ denote the exceptional divisor (i.e. the $(-1)$-curve), which is the preimage of the blown up point $P_i$.
Let $e_0$ be the divisor class of a line, which contains none of the points $P_i$, and let $e_i$ $(i=1,\ldots,4)$ be a divisor class of $E_i$. These classes generate a Picard group $\Pic Y=\bangle{e_0,\ldots,e_4}_\ZZ\cong \ZZ^{5}$. The intersection index defines a symmetric bilinear form on the Picard group such that the basis $\{e_0,\ldots,e_4\}$ is orthogonal, $e_0^2=1$ and $e_i^2=-1$.
Exceptional divisor classes are $e_i$ and $e_0-e_i-e_j$ for distinct $i, j\neq0$.

By Kleiman's ampleness criterion \cite[Theorem 1.4.9]{Laz} the closure of the ample cone $\Ample Y$ is dual to the cone of effective divisors $\Eff(Y)$. In case of a del Pezzo surface of degree $<8$ the cone $\Eff(Y)$ is generated by exceptional divisors \cite[Theorem 8.2.19]{Dolg}. Therefore, the ample cone is defined by inequalities
\begin{align}
  x_0>0,\; x_i<&0\qquad i=1,\ldots,4,\\
  x_0+x_i+x_j>&0,\qquad 0\neq i\neq j\neq0,
\end{align}
where $(x_0,\ldots,x_4)\in\Pic Y$. It has the following ten extremal rays
\begin{equation}\label{extrays5}
  e_0,\;e_0-e_j,\;2 e_0-\sum_{i\neq 0} e_i,\;2 e_0-\sum_{i\neq 0,j} e_i \quad\mbox{ where } j=1,\ldots,4.
\end{equation}

For five of them the corresponding orthogonal facet of the effective cone contains four non-intersecting (-1)-curves. They define the contraction $Y\to \PP^2$ corresponding to the chosen extremal ray.

Any other ray defines a pencil of quadrics on $Y$. More precisely, an orthogonal complement to the ray contains three pairs of intersecting (-1)-curves which define the degenerate fibers of the pencil. Herewith, the class of the pencil fiber belongs to the chosen ray.

\subsection{Cylinders}\label{3.1} 
Let us fix a blowdown $\varphi\colon Y\to \PP^2$ of four pairwise disjoint $(-1)$-curves $E_1,\ldots,E_4$ into points $P_1,\ldots, P_4$ using the notation as above. Let $l_{ij}\subset\PP^2$ be the line passing through the points $P_i$ and $P_j$. Consider the open subset $U_1=\varphi^{-1}(\PP^2\setminus(l_{12}\cup l_{34}))\subset Y$. This is a cylinder defined by the pencil of lines passing through the base point $\Bs(U_1)=l_{12}\cap l_{34}$. We have $U_1\cong \A^1_*\times\A^1$, where $\A^1_*=\A^1\setminus\{0\}$.
Similarly let $U_2=\varphi^{-1}(\PP^2\setminus(l_{13}\cup l_{24}))$ and $U_3=\varphi^{-1}(\PP^2\setminus(l_{14}\cup l_{23}))$, see fig.~\ref{graph5-3}. Furthermore, consider the blowings down of other 4-tuples of non-intersecting $(-1)$-curves on $Y$. There are five of them as shown on fig.~\ref{graph5-5}. For every blowing down we define three cylinders in a similar way. Note that these cylinders are in one-to-one correspondence with the intersection points of the (-1)-curves, and the automorphism group $\Aut Y\cong S_5$ acts transitively on them.
\begin{figure}[h]
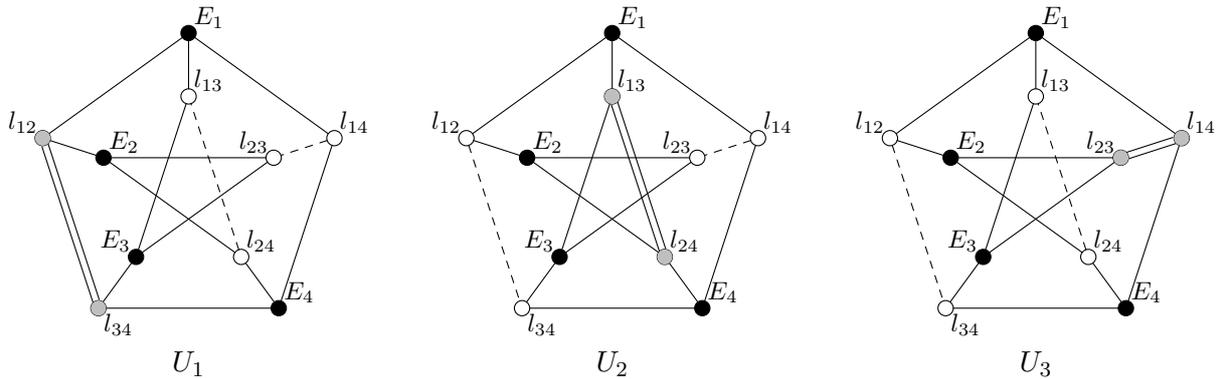

\begin{minipage}[h]{0.3\linewidth}
\center{\includegraphics[width=1.\linewidth]{petersen.0} \\ $U_1$}
\end{minipage}
\hfill
\begin{minipage}[h]{0.3\linewidth}
\center{\includegraphics[width=1.\linewidth]{petersen.1} \\ $U_2$}
\end{minipage}
\hfill
\begin{minipage}[h]{0.3\linewidth}
\center{\includegraphics[width=1.\linewidth]{petersen.2} \\ $U_3$}
\end{minipage}
\caption{ Arrangement of cylinders on the incidence graph of $(-1)$-curves on the del Pezzo surface of degree 5. The gray and the black vertices correspond to $(-1)$-curves forming the complement to a cylinder. The dashed edges correspond to $(-1)$-curve intersections contained in the cylinder. The double edge corresponds to the base point of the cylinder.}
\label{graph5-3}
\end{figure}
\begin{figure}[h]
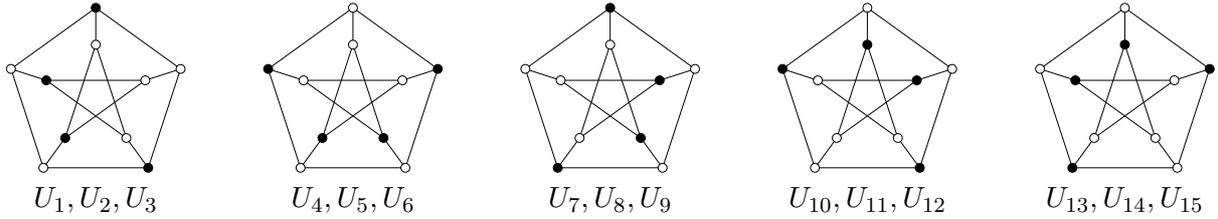

\begin{minipage}[h]{0.15\linewidth}
\center{\includegraphics[width=1.\linewidth]{petersen.3} \\ $U_1,U_2,U_3$}
\end{minipage}
\hfill
\begin{minipage}[h]{0.15\linewidth}
\center{\includegraphics[width=1.\linewidth]{petersen.4} \\ $U_4,U_5,U_6$}
\end{minipage}
\hfill
\begin{minipage}[h]{0.15\linewidth}
\center{\includegraphics[width=1.\linewidth]{petersen.5} \\ $U_7,U_8,U_9$}
\end{minipage}
\hfill
\begin{minipage}[h]{0.15\linewidth}
\center{\includegraphics[width=1.\linewidth]{petersen.6} \\ $U_{10},U_{11},U_{12}$}
\end{minipage}
\hfill
\begin{minipage}[h]{0.15\linewidth}
\center{\includegraphics[width=1.\linewidth]{petersen.7} \\ $U_{13},U_{14},U_{15}$}
\end{minipage}
\caption{Black vertices correspond to 4-tuples of $(-1)$-curves. Every blowing down defines three cylinders similarly as on fig.~\ref{graph5-3}.}
\label{graph5-5}
\end{figure}

 Thus we have cylinders $U_1,U_2,\ldots,U_{15}$ as shown on Figures \ref{graph5-3} and \ref{graph5-5}. It is easy to check that every intersection of $(-1)$-curves is contained in some cylinder, hence $\bigcup U_i=Y$. We claim that there is no proper subset $W\subset Y$, which is invariant with respect to all 15 cylinders. Assume that there exists such a subset $W$. Let us fix an arbitrary point of $W$. It is contained in a fiber $S$ of some cylinder, hence $W$ contains $S$. Without loss of generality $S$ is a fiber of $U_1$. Then the line $l=\overline{\varphi(S)}\subset\PP^2$ passes through the base point $\Bs(U_1)$. Since the points $\Bs(U_1), \Bs(U_2)$, and $\Bs(U_3)$ do not lie on the same line, one of them does not belong to~$l$. Suppose $\Bs(U_2)\notin l$. Then the fiber $S$ intersects almost every fiber of the cylinder $U_2$, and $W$ contains them. So, $W$ is dense in $Y$. The complement $Y\setminus W$ is also invariant with respect to all cylinders, and by the same reason it is dense in $Y$, a contradiction.

\subsection{Polarity condition}
\label{3.2}
In this subsection we show that for any ample divisor $H$ on $Y$ all the 15 cylinders $U_i$ are $H$-polar.
 Consider the set of effective divisors $\{\alpha_i E_i+\beta_{1}l_{12}+\beta_3 l_{34}\;|\; \alpha_i,\beta_i>0\}$ whose support is the complement to $U_1$. The image of this set in the Picard group is an open cone $C$, whose extremal rays are $e_1,\; e_2,\; e_3,\; e_4,\;  e_0-e_1-e_2,$ and $e_0-e_3-e_4$.
It is easy to check that the primitive vectors of the ample cone (\ref{extrays5}) can be expressed as linear combinations with non-negative rational coefficients of the primitive vectors of the cone $C$. Therefore the cylinder $U_1$ is $H$-polar for any ample divisor $H$. By automorphisms $\Aut Y$ we may translate $U_1$ to any cylinder, hence the cylinders $U_i$ are $H$-polar for any ample divisor $H$.
Using Theorem \ref{flex} we obtain the assertion. Now Theorem \ref{theo5} is proved.

\section{Del Pezzo surfaces of degree 4}
Every del Pezzo surface of degree 4 is isomorphic to a blowing up of a projective plane $\PP^2$ in five points, where no three are collinear. Such surfaces form a two-parameter family.

By $E_i$ we denote the $(-1)$-curve which is the preimages of the blown up point $P_i$.
As before, let $e_0$ be the divisor class of a line which does not contain the blown up points, and $e_i\;(i=1,\ldots,5)$ be the divisor class of $E_i$. A set $\{e_0,\ldots,e_5\}$ forms an orthogonal basis of the Picard group $\Pic Y\cong\ZZ^6$, and $e_0^2=1,\; e_i^2=-1$.
The classes of $(-1)$-curves are $e_i,\;e_0-e_i-e_j,\;2e_0-\sum_{k\neq0}e_k$ for any pair of distinct indices $i, j\neq0$.
The ample cone is defined by inequalities
\begin{align}
  x_0>0,\; x_i<&0\qquad i=1,\ldots,5,\\
  x_0+x_i+x_j>&0,\qquad 0\neq i\neq j\neq0, \\
  2x_0+x_1+\ldots+x_5>&0,
\end{align}
where $(x_0,\ldots,x_5)\in\Pic Y$. Its extremal rays are
\begin{equation}\label{extrays4}
  e_0,\;e_0-e_j,\;2 e_0-\sum_{k\neq 0,i} e_k,\;2 e_0-\sum_{k\neq 0,i,j} e_k,\mbox{ and }3 e_0-\sum_{k\neq 0} e_k-e_i
\end{equation}
for any pair of distinct indices $i,j\in\{1,\ldots,5\}$. Similarly to the case of del Pezzo surface of degree 5, sixteen extremal rays correspond to blowings down $Y\to \PP^2$, and ten rays correspond to pencils of quadrics on $Y$.
\subsection{Cylinders}
Let us fix some $(-1)$-curve $C_1$ and consider the blowing down $\sigma_1\colon Y\to\PP^2$ of five $(-1)$-curves $F_1,\ldots,F_5$ that meet $C_1$, see fig.~\ref{graph4}. This blowing down is well defined since the contracted divisors do not intersect.
The image $\sigma_1(C_1)$ is a smooth quadric $c$ passing through the blown down points $Q_1,\ldots,Q_5$.
Take an arbitrary line $l\subset\PP^2$ which is tangent to $c$ at a point  different from $Q_1,\ldots,Q_5$. A quadric pencil in $\PP^2$ generated by divisors $c$ and $2l$ determines a cylinder $U\cong\A^1_*\times\A^1\subset Y$ whose complement is the complete 	preimage of the support of the divisor $c+2l$ on $\PP^2$. Denote by $\U_{C_1}$ the family of all such cylinders in $Y$ for all such tangents $l$.
 Note that $Y\setminus \bigcup_{U\in \U_{C_1}}U$ is a union of $C_1$ and the exceptional divisors $F_i \;(i=1,\ldots,5)$.
Apply this construction to the $(-1)$-curves $C_2,\ldots,C_5$, which form a 5-cycle along with $C_1$ on the incidence graph as shown on fig.~\ref{graph4}. Overall we obtain five cylinder families $\U_{C_1},\ldots,\U_{C_5}$. It is easy to see that their union covers $Y$.

 Let $W$ be a proper subset of $Y$ which is invariant with respect to the cylinders of all families, and let $w\in W$ be an arbitrary point. We may suppose that $w$ belongs to a cylinder of the family $\U_{C_1}$. Then the image $\sigma_1(W)\subset\PP^2$ is invariant with respect to the cylinder family $\{\sigma_1(U)\;|\; U\in \U_{C_1}\}$. Note that every cylinder of this family is a complement to the quadric $c$ and its tangent line. It is well known that given a quadric and two points outside it we can find a quadric passing through these two points and tangent to the given quadric. Therefore, for almost every point $x\in \PP^2\setminus c$ there exists a fiber of some cylinder which contains $x$ and $\sigma_1(w)$. Namely, $x$ must not lie on the tangent line to $c$ passing through $\sigma_1(w)$ as well as on the quadrics which are tangent to $c$ at blown down points and contain $\sigma_1(w)$. Thus $W$ is dense in $Y$. Similarly, $Y\setminus W$ is dense in $Y$, a contradiction. Finally, the families $\U_{C_1},\ldots,\U_{C_5}$ form a transversal cover of~$Y$.
\begin{figure}[h]
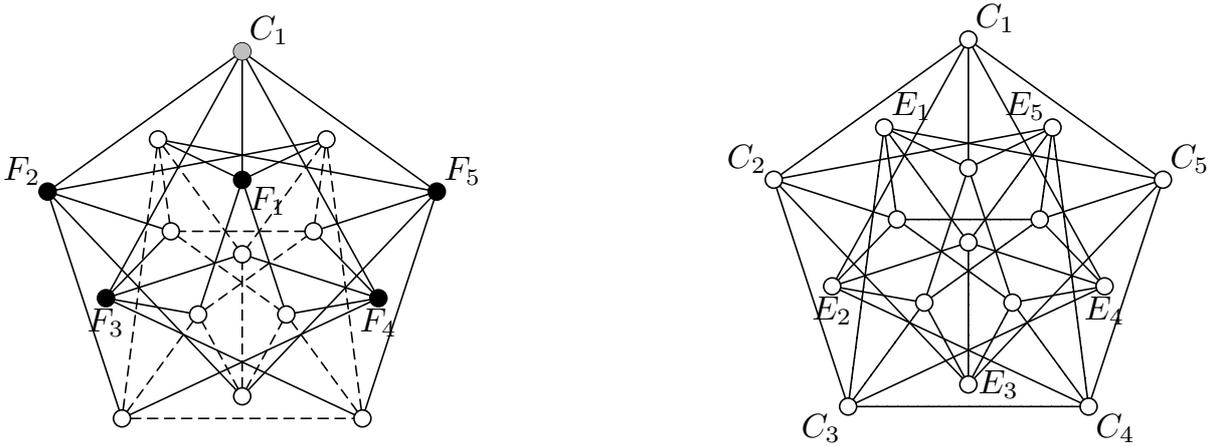

\begin{minipage}[h]{0.4\linewidth}
\center{\includegraphics[width=1.\linewidth]{clebsch.1}}
\end{minipage}
\hfill
\begin{minipage}[h]{0.4\linewidth}
\center{\includegraphics[width=1.\linewidth]{clebsch.0}}
\end{minipage}
\caption{ The incidence graph of $(-1)$-curves on a del Pezzo surface of degree 4. On the left the gray vertex corresponds to the quadric preimage $C_1$ and black vertices correspond to the contracted $(-1)$-curves. The dashed edges correspond to $(-1)$-curve intersections contained in the cylinders of a family. Four other families corresponding to $C_2,\ldots, C_5$ are obtained symmetrically by the graph rotations.}
\label{graph4}
\end{figure}

\subsection{Polarity condition}
 Ample divisors $H$ such that cylinders of the family $\U_{C_1}$ are $H$-polar, are exactly the ample divisors in the open cone $\Ample Y\cap\{\alpha_1F_1+\ldots+\alpha_5F_5+\alpha_6C_1+\alpha_7\sigma_1^{-1}(l)\;|\;\alpha_j>0\}$ in $\Pic Y$. We define such a cone for every $C_i,\, i=1,\ldots,5$ and denote it by $\Ample(C_i,Y)$. It does not depend on a choice of a tangent line $l$ since it does not contain blown up points by definition.
Then the set of divisors $H$ such that cylinders in $\bigcup_i \U_{C_i}$ are $H$-polar is an open cone $\bigcap_i \Ample(C_i,Y)$. A computation shows that it has exactly 72 extremal rays, which can be expressed as
{\allowdisplaybreaks
\begin{align*}
e_0&,&9e_0-5e_{i_1}-e_{i_2}-2e_{i_3}-4e_{i_4}-3e_{i_5},\\
4e_0&-2e_{i_1}-2e_{i_2}-e_{i_3}-e_{i_4}-e_{i_5},&9e_0-4e_{i_1}-4e_{i_2}-4e_{i_3}-2e_{i_4}-2e_{i_5},\\
5e_0&-2e_{i_1}-2e_{i_2}-e_{i_3}-3e_{i_4}-e_{i_5},&11e_0-6e_{i_1}-2e_{i_2}-2e_{i_3}-4e_{i_4}-4e_{i_5},\\
5e_0&-2e_{i_1}-2e_{i_2}-2e_{i_3}-2e_{i_4},&11e_0-6e_{i_1}-4e_{i_2}-4e_{i_3}-2e_{i_4}-2e_{i_5},\\
5e_0&-2e_{i_1}-2e_{i_2}-2e_{i_3}-2e_{i_4}-2e_{i_5},&11e_0-6e_{i_1}-2e_{i_2}-4e_{i_3}-4e_{i_4}-4e_{i_5},\\
6e_0&-2e_{i_1}-2e_{i_2}-3e_{i_3}-e_{i_4}-3e_{i_5},&11e_0-6e_{i_1}-4e_{i_2}-4e_{i_3}-4e_{i_4}-2e_{i_5},\\
7e_0&-4e_{i_1}-2e_{i_2}-2e_{i_3}-2e_{i_4}-2e_{i_5},&15e_0-8e_{i_1}-2e_{i_2}-4e_{i_3}-6e_{i_4}-6e_{i_5},\\
9e_0&-5e_{i_1}-3e_{i_2}-4e_{i_3}-2e_{i_4}-1e_{i_5},&15e_0-8e_{i_1}-6e_{i_2}-6e_{i_3}-4e_{i_4}-2e_{i_5},
\end{align*}}
where the tuple $(i_1,\ldots,i_5)$ runs over all cyclic permutations of $(1,2,3,4,5)$.

It is easy to see that the anticanonical divisor $(-K_Y)$ is contained in $\bigcap_i \Ample(C_i,Y)$.
Similarly to Theorem \ref{theo5} we obtain the following result.
\begin{theorem}
     Let $Y$ be a del Pezzo surface of degree 4, and $H$ be a very ample divisor in the open cone $\bigcap_{i=1}^5 \Ample(C_i,Y)$. Then the affine cone $\AffCone_H Y$ is flexible. In particular, this holds for the anticanonical divisor $H=-K_Y$.
\end{theorem}

  We have identified a subcone of the ample cone such that the very ample divisors contained in this subcone define a flexible affine cone. However, this subcone is strictly contained in the ample cone. For example, the ample divisor class $8e_0-2e_1-4e_2-e_3-e_4-3e_5$ lies outside of that subcone. Thus the flexibility problem for the affine cone over the polarization of a del Pezzo surface of degree~4 by any very ample divisor remains open.

\end{document}